\documentclass[12pt,fleqn]{amsart}

\usepackage{enumerate}
\usepackage{comment}
\usepackage{etex}
\usepackage[latin1]{inputenc}
\usepackage{amsmath}
\usepackage{amsfonts}
\usepackage{amssymb}
\usepackage{amsmath,amssymb,amsthm,amsfonts,graphicx,amscd} 
\usepackage{mathtools}
\usepackage{color}                    
\usepackage{hyperref}                 
\usepackage[T1]{fontenc}
\usepackage{tikz}
\usetikzlibrary{arrows}

\newtheorem{THEO}{Theorem}

\newtheorem{LEMM}[THEO]{Lemma}

\newtheorem{DEFI}[THEO]{Definition}

\newtheorem{QUES}[THEO]{Question}

\DeclarePairedDelimiterX{\norm}[1]{\lVert}{\rVert}{#1}

\newtheorem*{THEO*}{Theorem}
\newtheorem*{DEFI*}{Definition}

\def\({\left(}
\def\){\right)}

\def\N{\mathbb{ N}}

\newcommand{\sub}{\subseteq}

\def\CC{\mathcal{C}}

\def\EE{\mathcal{E}}
\def\FF{\mathcal{F}}

\def\1{\textbf{1}}

\def\co{\operatorname{co}}

\newcommand{\Suc}[2]{\ensuremath{\left({#1}\right)_{{#2}=1}^\infty}}

\def\w{\omega}

\title[]{On the property (C) of Corson and other sequential properties of Banach Spaces}

\author[Mart\'inez-Cervantes]{Gonzalo Mart\'inez-Cervantes}
\address[Mart\'inez-Cervantes]{Universidad de Alicante, Departamento de Matem\'{a}ticas, Facultad de Ciencias, 03080 Alicante, Spain
	\newline
	\href{http://orcid.org/0000-0002-5927-5215}{ORCID: \texttt{0000-0002-5927-5215} } }	

\email{gonzalo.martinez@ua.es}

\author[Poveda]{Alejandro Poveda}
\address[Poveda]{Harvard University, Center of Mathematical Sciences and Applications, Cambridge, MA 02138, USA
	\newline
	\href{https://scholar.harvard.edu/apoveda/home}{Website: https://scholar.harvard.edu/apoveda/home} }
\email{alejandro@cmsa.fas.harvard.edu}


\thanks{The first author was partially supported by Fundaci\'{o}n S\'{e}neca - ACyT Regi\'{o}n de Murcia (grant 21955/PI/22) and by Agencia Estatal de Investigaci\'on and EDRF/FEDER ``A way of making Europe" (MCIN/AEI/10.13039/501100011033) (grant PID2021-122126NB-C32). The second author is funded by the Center of Mathematical Sciences and Applications at Harvard University.
}

\begin{document}

\begin{abstract}
A well-known result of R. Pol states that a Banach space $X$ has property ($\CC$) of Corson if and only if every point in the weak*-closure of any convex set $C \subseteq B_{X^*}$ is actually in the weak*-closure of a countable subset of $C$.
Nevertheless, it is 
an open problem whether this is in turn equivalent to the countable tightness of $B_{X^*}$ with respect to the weak*-topology. Frankiewicz, Plebanek and Ryll-Nardzewski provided an affirmative answer under $\mathrm{MA}+\neg \mathrm{CH}$ for the class of $\CC(K)$-spaces. In this article we provide a partial extension of this latter result by showing that under the Proper Forcing Axiom  ($\mathrm{PFA}$) 
the following conditions are equivalent for an arbitrary Banach space $X$:
\begin{enumerate}
	\item $X$ has property $\EE'$; 
	\item $X$ has weak*-sequential dual ball; 
	\item $X$ has property ($\CC$) of Corson;
	\item $(B_{X^*},w^\ast)$ has countable tightness.
\end{enumerate}
This provides a partial extension of a former result of Arhangel'skii. 
In addition,  we show that every Banach space with property $\EE'$ has weak*-convex block compact dual ball. 
\end{abstract}

\maketitle

\section{Introduction}

A Banach space $X$ is said to have \emph{property ($\CC$) of Corson}\index{property ($\CC$) of Corson} if every family of closed convex subsets of $X$ with empty intersection contains a countable subfamily with empty intersection. R. Pol discovered that property ($\CC$) is a convex analogue of countable tightness:

\begin{THEO*}[Pol \cite{Pol80}]
	A Banach space $X$ has property $(\CC)$ if and only if every point in the weak*-closure of a convex set $C \sub B_{X^\ast}$ is in the weak*-closure of a countable subset of $C$.
\end{THEO*}

Recall that a topological space $T$ is said to have \textit{countable tightness} if for every subspace $F$ of $T$, every point in the closure of $F$ is in the closure of a countable subspace of $F$. Thus, if the dual ball of a Banach space $X$ has countable tightness (with respect to the weak*-topology) then $X$ has property $(\CC)$ of Corson. A natural question is whether the reverse implication holds.
To the best of our knowledge, this question is open (see \cite{FPR00, PleSob15, KKLP}).

Both countable tightness and property $(\CC)$ of Corson are closely related to several classical sequential properties. Before diving into deeper waters let us fix the notation and state some of the topological notions under consideration.
All topological spaces considered in the manuscript are assumed to be Hausdorff. The symbol $w^\ast$ will be reserved to denote the weak* topology of a Banach space.
Recall that a topological space $T$ is said to be \textit{sequentially compact} if every sequence in $T$ contains a convergent subsequence; whereas it is said to be \textit{Fr\'{e}chet-Urysohn} (FU for short) if for every subspace $F$ of $T$, every point in the closure of $F$ is the limit of a sequence in $F$. Thus, every FU compact space is sequentially compact and has countable tightness.
For our purpose and to follow the standard terminology in the field, let us just say that a Banach space is said to have \textit{weak*-angelic dual} if it has weak*-FU dual ball.

\smallskip

In this article we shall pay particular attention to a property which is between being FU and having countable tightness.
A topological space $T$ is said to be \textit{sequential} if any sequentially closed subspace is closed. Recall that a set is said to be \textit{sequentially closed} if it coincides with its sequential closure, i.e. with the set of all limits of sequences in the set.
 It is immediate that every FU space is sequential and that every sequential compact space is sequentially compact.
It is a well-known fact that every sequential space has countable tightness. However, whether the converse implication holds for the class of compact spaces is known as the \emph{Moore-Mrowka Problem}, which is undecidable on the grounds of ZFC. On the one hand, Balogh proved in \cite{Ba88} that under the \emph{Proper Forcing Axiom} (PFA) a compact space has countable tightness if and only if it is sequential. On the other hand, Fedorchuk \cite{Fedorchuk77} provided a consistent example of a compact space with countable tightness with no convergent subsequences, hence non-sequentially compact. Later variations of Balogh's theorem not bearing on PFA (hence, bypassing the use of large cardinal) have been obtained by Dow \cite{Dow92} and, more recently, by Dow and Eisworth \cite{DE}.

Thereby for a compact space $K$ we have the following implications:\bigskip

\begin{center}
\begin{tikzpicture}[->,>=stealth',shorten >=1pt,auto,node distance=3.5cm,
thick,main node/.style={rectangle,
	rounded corners,
	draw=black,
	text width=7.1em,
	text centered}]

\node[main node] (FU) {$K$ is FU};
\node at (2,0) {$\Rightarrow$};

\node[main node] (sequential) at (4,0) {$K$ is sequential};

\node at (6,0) {$\Rightarrow$};

\node[style={rectangle,rounded corners,draw=black,text width=12.1em, text centered}] (seqcomp) at (9,0) {$K$ is sequentially compact};

\node at (4,-0.75) {$\Downarrow$};

\node[style={rectangle,rounded corners,draw=black,text width=12.1em, text centered}] (counttightness) at (4,-1.5) {$K$ has countable tightness};

\end{tikzpicture}
\end{center}

In what follows we focus on  convex analogues of the above-mentioned properties when the topological space is the dual ball of a Banach space endowed with the weak*-topology.  
A Banach space $X$ is said to have property $\EE$ if every point in the weak*-closure of any \emph{convex} subset $C \subseteq B_{X^\ast}$ is the weak*-limit of a sequence in $C$ \cite[p.~352]{pli-yos-2}  . 
In the same spirit we say that $X$ has \emph{property $\EE'$} if every weak*-sequentially closed \emph{convex} set in the dual ball is weak*-closed \cite{MC17}.
Thus, if $X$ has weak*-angelic dual then it has property $\EE$, which in turn implies property $\EE'$. Likewise, if $X$ has weak*-sequential dual ball then $X$ has property $\EE'$.

The last property considered in this article is a convex version of sequential compactness:

\begin{DEFI*}
	\label{DEFIconvexblockcompact}
	If $(x_n)_{n}$ is a sequence in a Banach space, we say that $(y_k)_{k\in\N}$ is a \textit{convex block subsequence}\index{convex block subsequence} of $(x_n)_{n\in\N}$ if there is a sequence $(I_k)_{k\in \N}$ of finite subsets of $\N$  with $\max(I_k)< \min(I_{k+1})$ and a sequence $a_n \in [0,1]$ with $\sum_{n \in I_k} a_n = 1$ for every $k\in \N$ such that
	$$y_k = \sum_{n \in I_k} a_n x_n.$$

	A Banach space $X$ is said to have \textit{weak*-convex block compact dual ball}\index{weak*-convex block compact dual ball} if every bounded sequence in $X^\ast$ has a weak*-convergent convex block subsequence.
\end{DEFI*}

Since every subsequence is a convex block subsequence, it follows that every Banach space with weak*-sequentially compact dual ball has weak*-convex block compact dual ball.
A simple argument shows that 
every Banach space with property $\EE'$ has 
property $(\CC)$ (see Lemma \ref{LEMMETOC}). Furthermore, in 
Theorem \ref{THEOCONVEXsequentialimpliesconvexblockcpct} we prove that every Banach space with property $\EE'$ has weak*-convex block compact dual ball. Thus, the implications between these Banach-space properties can be summarized in a diagram as follows: 
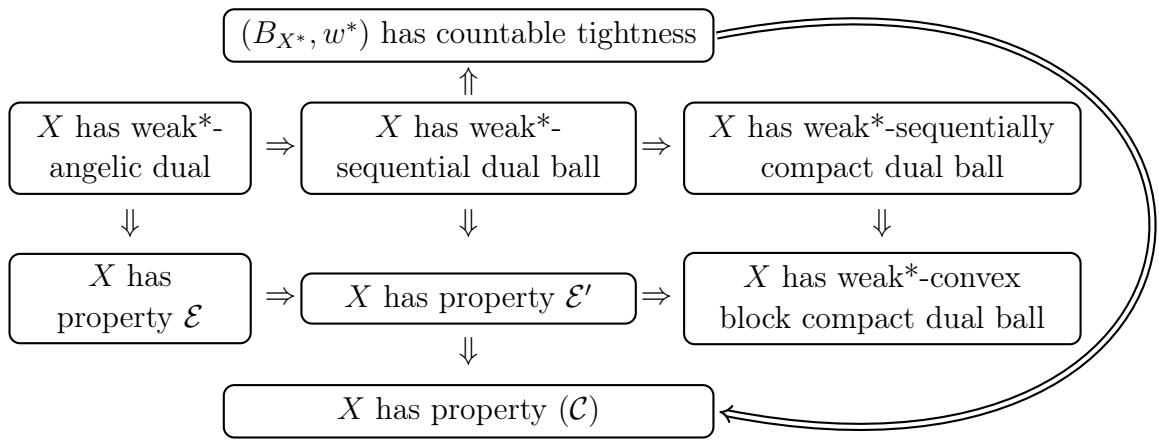
\vspace{-1cm}
\begin{center}
\begin{tikzpicture}[->,>=stealth',shorten >=1pt,auto,node distance=3.5cm,
thick,main node/.style={rectangle,
	rounded corners,
	draw=black,
	text width=7.1em,
	text centered}]

\node[main node] (FU) {$X$ has weak*-angelic dual};

\node at (2,0) {$\Rightarrow$};

\node[style={rectangle,rounded corners,draw=black,text width=10.1em, text centered}] (sequential) at (4.5,0) {$X$ has weak*-sequential dual ball};

\node at (7,0) {$\Rightarrow$};

\node[style={rectangle,rounded corners,draw=black,text width=12.1em, text centered}] (seqcomp) at (10,0) {$X$ has weak*-sequentially compact dual ball};

\node at (4.5,0.9) {$\Uparrow$};

\node[style={rectangle,rounded corners,draw=black,text width=15.1em, text centered}] (counttightness) at (4.5,1.5) {$(B_{X^*},w^*)$ has countable tightness};

\node[main node] at (0,-2) {$X$ has property $\EE$};

\node at (2,-2) {$\Rightarrow$};

\node[style={rectangle,rounded corners,draw=black,text width=10.1em, text centered}] (sequential) at (4.5,-2) {$X$ has property $\EE'$};

\node at (7,-2) {$\Rightarrow$};

\node[style={rectangle,rounded corners,draw=black,text width=12.1em, text centered}] (seqcomp) at (10,-2) {$X$ has weak*-convex block compact dual ball};

\node at (4.5,-1) {$\Downarrow$};

\node at (0,-1) {$\Downarrow$};

\node at (10,-1) {$\Downarrow$};

\node at (4.5,-2.7) {$\Downarrow$};

\node[style={rectangle,rounded corners,draw=black,text width=15.1em, text centered}] (counttightness) at (4.5,-3.5) {$X$ has property $(\CC)$};

\draw[-implies,double equal sign distance] (7.8,1.5) .. controls (15.5,3) and (15.5,-5).. (7.8,-3.5) ;

\end{tikzpicture}
\end{center}

A. Plichko and D. Yost asked in \cite[p.~352]{pli-yos-2} whether property ($\CC$) implies property $\EE$. J. Moore provided a negative answer in an unpublished notes using a modification of Ostaszewski space (a scattered compact space of height $\w_1$ with countable tightness, but not sequential) using $\Diamond$. 
In \cite{brech} C. Brech constructed yet another consistent example of a $\CC(K)$-space with property ($\CC$) but not property $\EE$. More specifically, Brech defines a scattered compact space of the form $K=L \cup \{ \infty \}$, where $K$ is the Alexandroff compactification of a space $L$ with the property that  $\ell_1(L)$ is weak*-sequentially closed in $\CC(K)^* = \ell_1(K)$.
Thereby, Brech's construction yields an example of a Banach space with property ($\CC$) that does not satisfy property $\EE'$.

\smallskip

In this article we prove the consistency of the converse implication. More specifically, our main result reads as follows.
\begin{THEO}[PFA]
\label{TheoremA}
The following conditions are equivalent for a Banach space $X$:
\begin{enumerate}
	\item $X$ has property $\EE'$;
	\item $X$ has weak*-sequential dual ball;
	\item $X$ has property $(\CC)$ of Corson;
	\item $(B_{X^*},w^*)$ has countable tightness.
\end{enumerate} 
\end{THEO}

Combining the above theorem with Brech's result from \cite{brech} we have that  the implication $(3)\implies(1)$ is independent of ZFC. On the contrary, it seems to be unknown whether $(3)\implies (4)$ is provable in ZFC alone. 
Frankiewicz, Plebanek and Ryll-Nardzewski showed that under $\mathrm{MA}+\neg \mathrm{CH}$ this implication holds provided $X$ is a $\CC(K)$-space \cite{FPR00}. We refer the reader to \cite{PleSob15} for more results and a comprehensive treatment of this problem. 
As far as we are concerned, it is unknow whether $(3)\implies (4)$ holds for arbitrary Banach spaces even under the extra assumption of $\mathrm{MA}+\neg \mathrm{CH}$. In general, we do not know whether Theorem \ref{TheoremA}  above holds true under the weaker assumption of $\mathrm{MA}+\neg\mathrm{CH}$.

Theorem~\ref{TheoremA} provides a partial extension of a  result of Arhangel'skii \cite{Arh} saying that, under the PFA, if $T$ is a Lindel\"{o}ff space then any compact subspace of  $\mathcal{C}_p(T)$\footnote{Namely, the space of continuous functions endowed with the topology of point-wise covergence.} has countable tightness. In particular, under the PFA, if a Banach space $X$ is  Lindel\"{o}ff with respect to its weak topology then $(B_{X^*},w^*)$ has countable tightness. Nevertheless the demand of  $X$ being  weak-Lindel\"{o}ff is way stronger than $X$ having property $(\CC)$ of Corson. Indeed, on the one hand weak-Lindel\"{o}ffness implies property $(\CC)$ of Corson\footnote{As every convex norm-closed set $F\subseteq X$ is weak-closed.}. On the other hand, there are Banach spaces $X$ with property $(\CC)$ of Corson which fail to be weak-Lindel\"{o}ff (see \cite[Example 2]{Corson} and recall that property $(\CC)$ is a three-space property).

\smallskip


\section{Property $\EE'$ and convex block subsequences}
\label{SectionE}

Let us start this section showing that property $\EE'$ implies property $(\CC)$.

\begin{LEMM}\label{LEMMETOC}
If a Banach space $X$ has property $\EE'$ then $X$ has property $(\CC)$ of Corson. 
\end{LEMM}
\begin{proof}
Let $S\subseteq B_{X^*}$ be a convex set and define 
    $$S_0:=\{x^* \in B_{X^*} \mid \mbox{there is a countable set } D\subseteq S \mbox{ such that } x^*\in \overline{D}^{w^*}\}.$$

    We have to prove that $\overline{S}^{w^*}=S_0$. Notice that $S_0$ is sequentially closed hence, as  $X$ has property $(\EE')$ and $S \subseteq S_0 \subseteq \overline{S}^{w^*}$, it suffices to show that $S_0$ is convex.
    Let $x^*,y^* \in S_0$ and $0\leq \lambda \leq 1$. Then, there exist countable sets $\{x_n^*\mid n\in\N\}$ and $\{y_n^*\mid n\in\N\}$ in $S$ such that $x^* \in \overline{\{x_n^*\mid n\in\N\}}^{w^*}$ and $y^* \in \overline{\{y_n^*\mid n\in\N\}}^{w^*}$. We claim that $$\lambda x^* + (1-\lambda)y^* \in \overline{\{\lambda x_n^* + (1-\lambda) y_m^* \mid n,m \in \N \}}^{w^*}.$$
    Let $U$ be any neighborhood of $\lambda x^* + (1-\lambda)y^*$. Then, there exist neighborhoods $U_1$ of $x^*$ and $U_2$ of $y^*$ such that 
    $$ \lambda U_1+(1-\lambda)U_2 :=\{\lambda z_1^* + (1-\lambda) z_2^* \mid z_1^* \in U_1,~z_2^* \in U_2 \} \subseteq U.$$
   Since $x^* \in \overline{\{x_n^*\mid n\in\N\}}^{w^*}$ and $y^* \in \overline{\{y_n^*\mid n\in\N\}}^{w^*}$, there are $n,m\in\N$ such that $x_n^* \in U_1$ and $y_m^* \in U_2$, so $\lambda x_n^* + (1-\lambda)y_m^* \in \lambda U_1+(1-\lambda)U_2 \subseteq U$, as desired.
\end{proof}

\bigskip

The goal of this section is to prove the following result, which was part of the first's author  Ph.D. thesis:
\begin{THEO}
	\label{THEOCONVEXsequentialimpliesconvexblockcpct}
	Every Banach space with property $\EE'$  has weak*-convex block compact dual ball. 
\end{THEO}

Given a Banach space $X$ and a set $C\subseteq B_{X^*}$ we denote by $S(C)$ its weak*-sequential closure; i.e. 
$S(C)$ is the set of limits of weak*-convergent sequences in $C$. It is immediate that $\overline{C}^{\| \cdot \|} \subseteq S(C)$ for every set $C \subseteq B_X$. Notice that $S(C)$ might not be sequentially closed. Nevertheless, the situation can be even worse, since the weak*-sequential closure of a bounded convex set $C$ might not be norm-closed (see \cite[Example 5.2]{AMCR}).
This fact motivates the following auxiliary lemma which will be used in the proof of Theorem~\ref{THEOCONVEXsequentialimpliesconvexblockcpct}.
\begin{LEMM}
	\label{LEMMAAUXILIARSEQUENTIALCLOSURENORMCLOUSURE}
	Let $X$ be a Banach space with property $\EE'$ and $C \sub B_{X^\ast}$ a convex subset. If the weak*-sequential closure of $C$ is equal to $\overline{C}^{\| \cdot \|}$ then $\overline{C}^{w^\ast}=\overline{C}^{\| \cdot \|}$.
\end{LEMM}
\begin{proof}
	Obviously $\overline{C}^{\| \cdot \|}$ is contained in $S(C)$. Suppose $S(C)=\overline{C}^{\| \cdot \|}$ but $\overline{C}^{w^\ast} \neq \overline{C}^{\| \cdot \|}$. Then $S(C)$ is not weak*-closed. Since $X$ has property $\EE'$, there exists a sequence $x_n^\ast \in S(C)$ which converges to a point $x^\ast \notin S(C)$. It follows from the equality $S(C)=  \overline{C}^{\| \cdot \|}$ that there exists $y_n^\ast \in C$ such that $\| y_n^\ast - x_n^\ast  \| < \frac{1}{2^n}$ for every $n \in \N$. Then,
	$ y_n^\ast = x_n^\ast +  (y_n^\ast - x_n^\ast ) $ is weak*-convergent to $x^\ast$. Thus $x^\ast \in S(C)$ in contradiction with our assumption.
\end{proof}

For the sake of completeness we include a proof of the following fact which is folklore (see, for instance, the proof of \cite[Corollary 6 in pg.~82]{DiestelUhl}).

\begin{LEMM}
\label{LEMMAAUXILIARMETRIZABILITY}
	Let $X$ be a Banach space and $A \sub B_{X^\ast}$ a norm-separable weak*-closed set. Then $A$ is weak*-metrizable.
\end{LEMM}
\begin{proof}
 Set $D=\{x_n^*\mid n\in\N\} \sub A$ a countable norm-dense set and fix  $x_{i,j} \in B_X$ for every  $i,j\in\N $ such that $\frac{\|x_i^*-x_j^*\|}{2} \leq x_i^*(x_{i,j})- x_j^*(x_{i,j})$. 
 If $x^\ast, y^\ast \in A$ are different then we can take $x^*_i,x^*_j \in D$ with $\| x^\ast - x_i^* \| < \frac{\|x^\ast- y^\ast\|}{8}$ and  $\| y^\ast - x_j^* \| < \frac{\|x^\ast- y^\ast\|}{8}$. Then $\|x_i^*-x_j^*\|>\frac{3\|x^*-y^*\|}{4}$ and
$$ x^*(x_{i,j})-y^*(x_{i,j}) > x_i^*(x_{i,j})-\frac{\|x^\ast- y^\ast\|}{8}- x_j^*(x_{i,j})-\frac{\|x^\ast- y^\ast\|}{8} \geq$$
$$ \geq  \frac{\|x_i^*-x_j^*\|}{2} - \frac{\|x^\ast- y^\ast\|}{4} >0.$$	
 
Thus, the set $\{x_{i,j} \mid i,j\in\N \}$ gives a countable family of weak*-continuous functions on the weak*-compact $A$ which separates points of $A$. Thus, $A$ is weak*-metrizable.
\end{proof}

\begin{proof}[Proof of Theorem \ref{THEOCONVEXsequentialimpliesconvexblockcpct}]
	Let $(x_n^\ast)_{n\in \mathbb{N}}$ be a sequence in $B_{X^\ast}$. We have to show that $(x_n^\ast)_{n\in\mathbb{N}}$ has a weak*-convergent convex block subsequence. Set $C$ the convex hull of $(x_n^\ast)_{n\in\mathbb{N}}$.

    Notice that if $S(C)=\overline{C}^{\| \cdot \|}$ then $S(C)=\overline{C}^{w ^\ast}=\overline{C}^{\| \cdot \|}$ by Lemma \ref{LEMMAAUXILIARSEQUENTIALCLOSURENORMCLOUSURE} and therefore $\overline{C}^{w ^\ast}$ is norm-separable, so weak*-metrizable by Lemma \ref{LEMMAAUXILIARMETRIZABILITY} and, in particular, $\Suc{x_n^\ast}{n}$ contains a weak*-convergent subsequence.

    Thus, we can assume that $S(C) \neq \overline{C}^{\| \cdot \|}$, i.e.~that there exists $x^\ast \in S(C) \setminus \overline{C}^{\| \cdot \|}$. Let $(y_n^\ast)_{n\in\mathbb{N}}$ be a sequence in $C$ weak*-convergent to $x^\ast$.
	Write $$ y_n ^\ast = \sum_{k=1}^\infty \lambda_k^n x_k ^ \ast,$$
	with $$\sum_{k=1}^\infty \lambda_k^n=1, ~ 0 \leq \lambda_k^n \leq 1 \mbox { and } \lambda_k^n=0 \mbox{ for all except finitely many } k \in \N$$
	for every $n\in \N$. Without loss of generality, we may suppose that $(\lambda_k^n)_{n\in\mathbb{N}}$ converges to some point $\lambda_k$ for every $k \in \N$. Moreover, a standard diagonal argument proves that we can assume that each sequence $(\lambda_k^n)_{n\in\mathbb{N}}$ is eventually monotone. Notice that $0 \leq \sum_{k=1}^\infty \lambda_k \leq 1$. We claim that $\sum_{k=1}^\infty \lambda_k <1$. If $\sum_{k=1}^\infty \lambda_k =1$ then $\sum_{k=1}^\infty \lambda_k x_k^\ast \in \overline{C}^{\| \cdot\|}$ and it can be easily seen that $x^\ast = \sum_{k=1}^\infty \lambda_k x_k^\ast $, in contradiction with $x^\ast \notin  \overline{C}^{\| \cdot\|}$.

	Thus, $0 \leq \lambda:=\sum_{k=1}^\infty \lambda_k <1$. Set $N_n = \{ k \in \N \mid \lambda_k^n > \lambda_k\}$ for every $n \in \N$. Notice  that
	\begin{equation}
	\label{equationaux3.1}
	\sum_{k \in N_n} (\lambda_k^n - \lambda_k ) = 1 - \sum_{k \in N_n} \lambda_k - \sum_{k \notin N_n } \lambda_k^n \geq 1 - \sum_{k=1}^\infty \lambda_k = 1 - \lambda >0.
	\end{equation}
	
	Passing to a subsequence if necessary, we suppose that there exists
	\begin{equation}
	\label{equationauxiliarlimit}
	\lambda' = \lim_n \sum_{k \in N_n} (\lambda_k^n - \lambda_k ) \geq 1- \lambda >0.
	\end{equation}
	
	Set $y_n^+ = \sum_{k \in N_n} \lambda_k^n x_k^\ast$ and $y_n^-= \sum_{k \notin N_n} \lambda_k^n x_k^\ast$.
	We claim that $(y_n^-)_{n\in\mathbb{N}}$ is norm-Cauchy and therefore norm-convergent.
	Fix $\varepsilon >0$ and take $k_0,k_1 \in \N$ such that $\sum_{k > k_0 } \lambda_k < \frac{\varepsilon}{4}$, $(\lambda_k^n)_{n\geq k_1}$ is monotone for every $k \leq k_0$ and $$ \sum_{k=1}^{k_0} |\lambda_k^n - \lambda_k^m| < \frac{\varepsilon}{2}$$ for every $n,m \geq k_1$.
	Then,
	
	$$ \|y_n^- - y_m^- \| \leq \sum_{k=1}^{k_0} |\lambda_k^n - \lambda_k^m| + \sum_{k > k_0,~k \notin N_n} \lambda_k^n + \sum_{k > k_0,~k \notin N_m}\lambda_k^m \leq \frac{\varepsilon}{2}+ 2 \sum_{k > k_0} \lambda_k < \varepsilon$$
	for every $n,m \geq k_1$.
	Thus, $(y_n^-)_{n\in\mathbb{N}}$ is norm convergent and, since $y_n^+ = y_n^\ast-y_n^-$, the sequence $(y_n^+)_{n\in\mathbb{N}}$ is weak*-convergent to a point $y^\ast$.

	Set 
	$$z_n^\ast = \frac{1}{\sum_{k \in N_n} (\lambda_k^n - \lambda_k )}  \sum_{k \in N_n}(\lambda_k^n - \lambda_k )x_k^\ast = \frac{1}{\sum_{k \in N_n} (\lambda_k^n - \lambda_k )} (y_n^+ - \sum_{k \in N_n}\lambda_k x_k^\ast )$$ for every $n \in \N$.  It follows from (\ref{equationaux3.1}) and (\ref{equationauxiliarlimit}) that $z_n^\ast$ is well-defined, $z_n^\ast \in C$ and it converges to $\frac{1}{\lambda'}(y^\ast - \sum_{k \in M} \lambda_k x_k^\ast)$, where $$M= \{ k \in \N\mid  (\lambda_k^n)_{n\in\mathbb{N}} \mbox{ is eventually decreasing}\}.$$
	
	\medskip
	
	For each $n,k \in \N$, write $\beta_k^n = \frac{\lambda_k^n - \lambda_k}{\sum_{k' \in N_n} (\lambda_{k'}^n - \lambda_{k'} )}$ if $k \in N_n$ and $\beta_k^n=0$ if $k \notin N_n$.
	Then, $(\beta_k^n)_{n\in\mathbb{N}}$ converges to zero and $z_n^\ast= \sum_{k=1}^\infty \beta_k^n x_k^\ast$.
	By taking small perturbations of $z_n^\ast$, we are going to construct a convex block subsequence $(u_n^\ast)_{n\in\mathbb{N}}$ of $(x_n^\ast)_{n\in\mathbb{N}}$ with the same limit than $(z_n^\ast)_{n\in\mathbb{N}}$.

	Fix $u_1^\ast=z_1^\ast$ and take a finite subset $I_1$ of $\N$ such that 
	$u_1^\ast= \sum_{k \in I_1} \beta_k^1 x_k^\ast$.
	Fix $n_1 \in \N$ such that $\sum_{k \in I_1} \beta_k^n < \frac{1}{2}$ for every $n \geq n_1$ and take 
	a finite set $I_2$ with $\max(I_1)< \min(I_2)$ such that
	$z_{n_1}^\ast = \sum_{k \in I_1 \cup I_2} \beta_k^{n_1} x_k^\ast$.
	Take $u_2^\ast =  \frac{1}{\sum_{k \in I_2} \beta_k^{n_1} } \sum_{k \in I_2} \beta_k^{n_1} x_k^\ast$.
	Notice that $u_2^\ast \in C$ and $$\| z_{n_1}^\ast - u_2^\ast \| < \frac{1}{2} + \sum_{k \in I_2} \beta_k^{n_1}\biggl(\frac{1}{\sum_{k' \in I_2} \beta_{k'}^{n_1}}-1\biggr) \leq \frac{1}{2} +  \biggl(\frac{1}{\sum_{k \in I_2} \beta_{k}^{n_1}}-1\biggr).$$

	Repeating this argument we construct a sequence $(u_n^\ast)_{n\in\mathbb{N}}$ in $C$, an increasing sequence $(n_k)_{k\in\mathbb{N}}$ in $\N$ and a sequence of finite sets $(I_k)_{k\in\mathbb{N}}$ of $\N$ with $\max(I_k) < \min(I_{k+1})$ such that
	$$ \sum_{k \in I_1\cup I_2 \cup \dots \cup I_{r}} \beta_k^m < \frac{1}{2^r} $$ for every $m \geq n_r$, 
	$$ z_{n_r}^\ast = \sum_{k \in I_1\cup I_2 \cup \dots \cup I_{r+1}} \beta_k^{n_r} x_k^\ast,$$
	$$u_{r+1}^\ast= \frac{1}{\sum_{k \in I_{r+1}} \beta_k^{n_r} } \sum_{k \in I_{r+1}} \beta_k^{n_r} x_k^\ast$$ 
	and 
	$$\| z_{n_r}^\ast - u_{r+1}^\ast \| < \frac{1}{2^r} + \sum_{k \in I_{r+1}} \beta_k^{n_r}\biggl(\frac{1}{\sum_{k' \in I_{r+1}} \beta_{k'}^{n_r}}-1\biggr) \leq \frac{1}{2^r} +  \biggl(\frac{1}{\sum_{k \in I_{r+1}} \beta_{k}^{n_r}}-1\biggr)$$
	$$\leq \frac{1}{2^r}+\biggl(\frac{2^r}{2^r-1}-1\biggr)=\frac{1}{2^r}+\frac{1}{2^r-1},$$	
	where the last inequality follows from 
	$$ \sum_{k \in I_{r+1}} \beta_{k}^{n_r} = 1 - \sum_{k \in I_1\cup I_2 \cup \dots \cup I_{r}} \beta_k^{n_r} \geq 1 - \frac{1}{2^r}=\frac{2^r-1}{2^r}.$$ 
	
	Since $\| z_{n_r}^\ast - u_{r+1}^\ast \|$ converges to zero, we conclude that $(u_r^\ast)_{r\in\mathbb{N}}$ is a convex block subsequence of $(x_n^\ast)_{n\in\mathbb{N}}$ which is weak*-convergent with the same limit than $(z_n^\ast)_{n\in\mathbb{N}}$.
\end{proof}

\section{From property $(\CC)$ to countable tightness}
\label{SectionEquivalences}

In this section we characterize (under PFA)  property $(\CC)$ of Corson (Theorem \ref{TheoremA}): 

\begin{LEMM}[PFA]
\label{TheoCtightness}
If $X$ has property $(\CC)$ of Corson then $(B_{X^*},w^*)$ has countable tightness.
\end{LEMM}
In our proof of Lemma \ref{TheoCtightness} we  use the following notion 
due to Balogh \cite{Ba88}:  
\begin{DEFI}
	Let $T$ be a non-empty set and $L, R \subseteq T^2$ two reflexive relations. A sequence $\langle x_\alpha\mid \alpha\in\omega_1\rangle\subseteq T$ of distinct points 
	is  an \emph{increasing $\omega_1$-sequence for  the pair $\langle L,R\rangle $} if
$$\forall \beta\in\omega_1\;(L(x_\beta)\cap \langle x_\alpha\mid \alpha<\omega_1\rangle\subseteq \langle x_\alpha\mid  \alpha\leq \beta\rangle  \subseteq R(x_\beta)\cap \langle x_\alpha\mid \alpha<\omega_1\rangle)\footnote{Here $L(x)$ (resp. $R(x)$) is a shorthand for $\{y\in T\mid (x,y)\in L\}$ (resp. $\{y\in T\mid (x,y)\in R\}$).}.$$ 
\end{DEFI} 
\begin{DEFI}
	Let  $T$ be a topological space and $\FF$ a closed filter on $T$\footnote{I.e., a  filter consisting of closed sets in $T$.}. 
 
 A subspace  $S\subseteq T$ is said to be \emph{$\FF$-small} if there is  $F \in \FF$ with $S \cap F=\emptyset$. A relation $R\subseteq T^2$ is called $\FF$-small-valued if $R(x)$ is $\FF$-small for every $x\in T$.
\end{DEFI}



\smallskip

The following theorem due to Balogh will be instrumental in our proof of Lemma~\ref{TheoCtightness}:

\begin{THEO}\cite[Theorem 1.1]{Ba88}
\label{TheoBalogh}
Let $T$ be a countably compact, noncompact $T_1$-space, and let $\FF$ be a free closed filter on $T$. Suppose that $L,R \subseteq T^2$ is a pair of reflexive relations  such that $L$ is $\FF$-small-valued and for every $x\in T$, $L(x)$ is closed and $R(x)$ is open.

Then, under $\mathrm{PFA}$, there is an increasing $\omega_1$-sequence for the pair $\langle L,R\rangle $.
\end{THEO}

\begin{proof}[Proof of Lemma \ref{TheoCtightness}]
Suppose that $(B_{X^*},w^*)$ does not have countable tightness. Then, there exists $S\subseteq B_{X^*}$ which is not weak*-closed but such that $\overline{A}^{w^*}\subseteq S$ for any countable set $A\subseteq S$. In particular, $(S,w^*)$ is countably compact.
Set $x^*\in \overline{S}^{w^*}\setminus S$ and consider $$\FF=\{\overline{U}^{w^*}\cap S\mid  U \mbox{ is a weak*-open convex set in }B_{X^*} \mbox{ with } x^* \in U\}.$$ 

Notice that $\FF$ is a closed filter in $(S,w^*)$ and $\bigcap_{B\in \FF} B=\emptyset$; i.e.,~$\FF$ is a free filter.
Thus, by Hahn-Banach Separation Theorem, for every $y \in S$ there are weak*-open convex neighborhoods (in $B_{X^*}$) $G(y)$ and $U(y)$ of $y$ such that $\overline{G(y)}^{w^*}\cap S$ is $\FF$-small and $$\overline{U(y)}^{w^*} \subseteq G(y).$$

Set $L(y):=\overline{G(y)}^{w^*}\cap S$ and $R(y):=U(y)\cap S$ for every $y\in S$.
We can apply Theorem \ref{TheoBalogh} with respect to  $T:=S$ and the pair $\langle L,R\rangle$ and obtain, in return, an increasing $\omega_1$-sequence $\langle x_\alpha\mid  \alpha<\omega_1\rangle$ in $S$ for $\langle L,R\rangle $.
Thus, by definition,
\begin{equation}
\label{eqrelations}
 \overline{\{x_\alpha: \alpha\leq \beta\}}^{w^*} \subseteq \overline{R(x_\beta)}^{w^*} \subseteq G(x_\beta) \subseteq B_{X^*}\setminus \{x_\alpha: \alpha>\beta\}
 \end{equation}
for every $\beta <\omega_1$. 
Now, set $L=\bigcup_{\beta<\omega_1} \overline{\co\{x_\alpha: \alpha<\beta\}}^{w^*}$. It is clear that $L$ is a convex set such that $\overline{C}^{w^*}\subseteq L$ for every countable set $C\subseteq L$. To finish the proof it is enough to show that $L$ is not weak*-closed, since in such case $X$ does not have property ($\CC$) of Corson.
Since each $G(x_\beta)$ is convex, by (\ref{eqrelations}), we have $$\overline{\co\{x_\alpha: \alpha<\beta\}}^{w^*} \subseteq G(x_\beta)$$ for every $\beta <\omega_1$ and, therefore, $\{G(x_\beta):\beta<\omega_1\}$ is a weak*-open cover of $L$. Nevertheless, again by (\ref{eqrelations}), there is no finite subcover for $L$. Thus, $L$ is not weak*-compact and, in particular, it is not weak*-closed, as desired.
\end{proof}
The time is now ripe to prove the main result of the section:
\begin{proof}[Proof of Theorem~\ref{TheoremA}]\hfill

    (1) $\Rightarrow$ (3): This is Lemma~\ref{LEMMETOC}.
    
    (3) $\Rightarrow$ (4): This is Lemma~\ref{TheoCtightness}.

    (4) $\Rightarrow$ (2): This is Balogh's main theorem in \cite{Ba88}.

    (2) $\Rightarrow$ (1): This is obvious. 
\end{proof}

\section{Further remarks and questions}
\label{SectionRemarksAndQuestions}

The question of whether property $(\CC)$ of Corson entails property $\EE$ was raised in \cite[p.~352]{pli-yos-2}. In the introduction we have mentioned  several (consistent) examples which answer this question in the negative. Besides, we have also shown that it is consistent for property  $(\CC)$ of Corson to imply property $\EE'$.
In \cite{AMCR} an example of a Banach space with property $\EE'$ but failing to have property $\EE$ was constructed under the CH. Nevertheless, the following question seems to be open.
\begin{QUES}
Is it consistent with $\mathrm{ZFC}$ that properties $\EE$ and  $\EE'$ coincide?
\end{QUES}

Furthermore, yet another example of a Banach space with property $\EE$ but without weak*-angelic dual was constructed in \cite{AMCR}, also under CH. Once again, it seems to be unknown whether it is consistent with ZFC that these two properties coincide.

Finally, concerning properties $\EE'$ and $(\CC)$, the following  seems to be open.
\begin{QUES}\hfill
\begin{enumerate}\item Does every Banach space with property $\EE'$ have weak*-sequential dual ball?
\item Does every Banach space with property $(\CC)$ of Corson have dual ball with countable tightness (with respect to the weak*-topology)?
\end{enumerate}
\end{QUES}

\section*{Acknowledgments}
We are grateful to J. Rodriguez for his feedback and careful reading of the manuscript which helped us to improve the paper.


\end{document}